\documentclass[12pt]{article}
\usepackage[latin1]{inputenc}
\usepackage{geometry}
\usepackage{amsfonts}
\usepackage{amsmath}
\usepackage{amssymb}
\usepackage{amsthm}
\usepackage{graphics}

\newcommand{\Q}{\mathbb{Q}}

\newcommand{\uni}{\cup}

\newcommand{\prc}[1]{\ensuremath{\mathsf{#1}}}

\newcommand{\SP}{\prc{\#P}}

\newcommand{\vc}[1]{\bar{#1}}  

\newcommand{\ee}{\tilde{e}}

\renewcommand{\AA}{\tilde{A}}
\newcommand{\BB}{\tilde{B}}
\newcommand{\Ge}{G_{ee}}

\newcommand{\subs}{\subseteq}

\newcommand{\vduni}{\sqcup}
\newcommand{\kc}{k_{cov}}
\newcommand{\yy}{\tilde{y}}
\newcommand{\vv}{\tilde{v}}
\newcommand{\GG}{\tilde{G}}

\newtheorem{defi}{Definition}

\newtheorem{lem}[defi]{Lemma}
\newtheorem{thm}[defi]{Theorem}

\begin{document}

\title{A Most General Edge Elimination Polynomial --- Thickening of Edges}

\author{Christian Hoffmann}

\maketitle

\begin{abstract}
  We consider a graph polynomial $\xi(G;x,y,z)$ introduced by
  Averbouch, Godlin, and Makowsky (2007). This graph polynomial
  simultaneously generalizes the Tutte polynomial as well as a
  bivariate chromatic polynomial defined by Dohmen, P\"onitz and
  Tittmann (2003). We derive an identity which relates the graph
  polynomial of a thicked graph (i.e.\ a graph with each edge replaced
  by $k$ copies of it) to the graph polynomial of the original graph.
  As a consequence, we observe that at every point $(x,y,z)$, except
  for points lying within some set of dimension~$2$, evaluating $\xi$
  is \SP-hard.
\end{abstract}

\section{Introduction}

We consider the following three-variable graph polynomial which has
been introduced by I.~Averbouch, B.~Godlin, and J.~A.~Makowsky
\cite{agm_most_general_edge_elimination}:
\begin{equation}
  \label{eq:xi_def}
  \xi(G;x,y,z) = \sum_{
    (A \sqcup B) \subseteq E}
  x^{k(A\uni B)-\kc(B)}\cdot y^{|A|+|B|-\kc(B)}\cdot
  z^{\kc(B)},
\end{equation}
where $G=(V,E)$ is a graph with multiple edges and self loops allowed,
$A\vduni B$ denotes a \emph{vertex-disjoint} union of edge sets $A$
and $B$, $k(A\uni B)$ is the number of components of $(V,A\uni B)$,
and $\kc(B)$ is the number of components of $(V(B),B)$.

The polynomial $\xi$ simultaneously generalizes two interesting graph
polynomials: the Tutte polynomial and a bivariate chromatic polynomial
$P(G;x,y)$ defined by K.~Dohmen, A.~P\"onitz, and P.~Tittmann
\cite{dpt_chromatic}.

It is known that the Tutte polynomial of a graph with ``thicked''
edges evaluated at some point equals the Tutte polynomial of the
original graph evaluated at another point (parallel edge reduction).
This property can be used to prove that at almost every point
evaluating the Tutte polynomial is hard \cite{jaeger_vertigan_welsh,
  sokal-2005, blaes_mak, bdm}.

In Section~\ref{sec:p2p} of this note we observe that edge thickening
has a similar effect on $\xi$ as on the Tutte polynomial
(Theorem~\ref{thm:p2p}). In Section~\ref{sec:xi_hard} we conclude that
for every point $(x,y,z) \in \Q^3$, except on a set of dimension at
most 2, it is \SP-hard to compute $\xi(G;x,y,z)$ from $G$
(Theorem~\ref{thm:xi_hard}). This supports a difficult point
conjecture for graph polynomials \cite[Conjecture 1]{makowsky_zoo},
\cite[Question 1]{agm_most_general_edge_elimination}.

\section{A point-to-point reduction from thickening}
\label{sec:p2p}

In this section we apply Sokal's approach to $\xi$ and obtain
Lemma~\ref{lem:psi_thick} (cf.\ \cite[Section 4.4]{sokal-2005}), the
main technical contribution of this note.

We define the following auxiliary polynomial, which has a different
$y$-variable for each edge of the graph, $\vc{y}=(y_e)_{e\in E(G)}$.
\begin{equation}
  \label{eq:xi_def}
  \psi(G;x,\vc{y},z) = \sum_{(A\vduni B) \subs E(G)} w(G;x,\vc{y},z;A,B),
\end{equation}
where
\[ w(G;x,\vc{y},z;A,B)=x^{k(A\uni B)}\big(\prod_{e\in (A\uni B)} y_e)
z^{\kc(B)}.\] We write $\psi(G;x,y,z)$ for $\psi(G;x,\vc{y}, z)$ if
for each $e\in E(G)$ we have $y_e=y$.

\begin{lem}
\label{lem:psi_xi}
We have the polynomial identities
$\psi(G;x,y,zx^{-1}y^{-1})=\xi(G;x,y,z)$ and
$\xi(G;x,y,zxy)=\psi(G;x,y,z)$.\qed
\end{lem}

Let $G$ be a graph and $e\in E(G)$ an edge. Let $E':=E\setminus \{e\}$
and $G_{ee}$ be the graph $G$ with $e$ doubled, i.e.\
$G_{ee}=(V(G),E'\uni\{e_1,e_2\})$ with $e_1,e_2$ being new edges.

\begin{lem}
\label{lem:psi_thick}
  $\psi(G_{ee};x,\vc{y},z)=\psi(G;x,\vc{Y}, z)$ with
  $Y_e=(1+y_{e_1})(1+y_{e_2})-1$ and $Y_{\ee}=y_{\ee}$ for all $\ee
  \in E'$.
\end{lem}
\begin{proof}
  Let $M(G)=\{(A,B)\ |\ A\vduni B \subs E(G)\}$ and $M(G_{ee}) =
  \{(\AA, \BB)\ |\ \AA\vduni \BB \subs E(G_{ee}\}$. We define a map
  $\tau:M(G) \to 2^{M(\Ge)}$ in the following way. Consider $(A,B)\in
  M(G)$. If $e\not \in A\uni B$, we set $\tau(A,B)=\{(A,B)\}$. If
  $e\in A$, we let $A':=A\setminus \{e\}$ and define
  $\tau(A,B)=\{(A'\uni \{e_1\}, B),(A'\uni \{e_2\},B),(A'\uni
  \{e_1,e_2\},B)\}$. (Note that in this case $e\not \in B$, as $A$ and
  $B$ are vertex-disjoint.) If $e \in B$, we let $B':=B\setminus
  \{e\}$ and define $\tau(A,B)=\{(A, B'\uni \{e_1\}), (A,B'\uni
  \{e_2\}), (A,B'\uni\{e_1,e_2\})\}$. Observe that
  \begin{equation}
    \label{eq:MGe_part}
    M(\Ge)=\uni_{(A,B)\in M(G)} \tau(A,B),
  \end{equation} and that this union is a
  union of \emph{pairwise disjoint} sets.

  Calculation yields
  \begin{equation}
  \label{eq:sum2subst}
  w(G;x,\vc{Y},z;A,B) =
  \sum_{(\AA,\BB)\in\tau(A,B)}w(\Ge,x,\vc{y},z;\AA,\BB)
  \end{equation}
  for every $(A,B)\in M(G)$. Thus,
  \begin{align*}
    \psi(\Ge;x,\vc{y},z) &= \sum_{(\AA,\BB)\in M(\Ge)}w(\Ge;x,\vc{y},z;\AA,\BB) \\
    &=\sum_{(A,B)\in M(G)}\sum_{(\AA,\BB)\in\tau(A,B)}w(\Ge;x,\vc{y},z;\AA,\BB) && \text{by \eqref{eq:MGe_part}}\\
    &=\sum_{(A,B)\in M(G)}w(G;x,\vc{Y},z;A,B) && \text{by \eqref{eq:sum2subst}}\\
    &=\psi(G;x,\vc{Y},z).
  \end{align*}
\end{proof}

Applying Lemma~\ref{lem:psi_thick} repeatedly and
Lemma~\ref{lem:psi_xi} to convert between $\psi$ and $\xi$ we obtain
\begin{thm}
\label{thm:p2p}
  Let $G_k$ be the $k$-thickening of $G$ (i.e.\ the graph obtained out
  of $G$ by replacing each edge by $k$ copies of it). Then
  \begin{align}
    \label{eq:psi_p2p}
  \psi(G_k;x,y,z)&=\psi(G;x,(1+y)^k-1,z), \\    
  \xi(G_k;x,y,z)&=\xi\Big(G;x,(1+y)^k-1,z\frac{(1+y)^k-1}{y}\Big).
  \end{align}
\end{thm}

\section{Hardness}
\label{sec:xi_hard}

The following theorem has been proven independently by I.~Averbouch
(J.~A.~Makowsky, personal communication, October 2007).
\begin{thm}
\label{thm:DPT_hard}
Let $P$ denote the bivariate chromatic polynomial defined by
K.~Dohmen, A.~P\"onitz, and P.~Tittmann \cite{dpt_chromatic}. For
every $(x,y)\in \Q$, $y\neq 0$, $(x,y)\not \in \{(1,1),(2,2)\}$, it is
\SP-hard to compute $P(G;x,y)$ from $G$.
\end{thm}
\begin{proof}[Proof (Sketch)]
  Given a graph $G=(V,E)$ let $\GG$ denote the graph obtained out of
  $G$ by inserting a new vertex $\vv$ and connecting $\vv$ to all
  vertices in $V$. Let $P(G;y)$ denote the chromatic polynomial
  \cite{read_chromatic}. It is well known that
  \begin{equation}
    \label{eq:chrom_p2p}
    P(\GG;y)=yP(G;y-1).
  \end{equation}
  From this and \cite[Theorem 1]{dpt_chromatic} we can derive
  \begin{equation}
    \label{eq:dpt_p2p}
    P(\GG;x,y)=yP(G;x-1,y-1)+(x-y)P(G;x,y).
  \end{equation}
  The proof of the theorem now works in the same fashion as a proof
  that $P(G;y)$ is \SP-hard to evaluate almost everywhere using
  \eqref{eq:chrom_p2p} would work: using \eqref{eq:dpt_p2p} we reduce
  along the lines $x=y+d$, which eventually enables us to evaluate $P$
  at $(1+d,1)$ (if $y$ is a positive integer, we reach $(1+d,1)$
  directly; otherwise we obtain arbitrary many points on the line
  $x=y+d$, which enables us to interpolate the polynomial on this
  line). On the line $y=1$ the polynomial $P$ equals the independent
  set polynomial \cite[Corollary 2]{dpt_chromatic}, which is \SP-hard
  to evaluate almost everywhere \cite{averbouch_makowsky,
    bh_interlace}.
\end{proof}

\begin{thm}
\label{thm:xi_hard}
For every $(x,y,z)\in \Q$, $x\neq 0$, $z\neq -xy$, $(x,z)\not \in
\{(1,0),(2,0)\}$, $y\not \in \{-2,-1,0\}$, the following statement
holds true: It is \SP-hard to compute $\xi(G;x,y,z)$ from $G$.
\end{thm}

\begin{proof}[Proof (Sketch)]
  For $x,y \in \Q$, $x, y\neq 0$ and $(x,y)\not \in \{(1,1),(2,2)\}$
  the following problem is \SP-hard by Theorem~\ref{thm:DPT_hard}:
  Given $G$, compute
  \[P(G;x,y)=\xi(G;x,-1,x-y) = \psi\Big(G;x,-1,\frac{y-x}{x}\Big),
  \]
  where the first equality is by \cite[Proposition
  18]{agm_most_general_edge_elimination} and the second by
  Lemma~\ref{lem:psi_xi}. We will argue that, for any fixed $\yy \in
  \Q\setminus \{-2,-1,0\}$, this reduces to compute
  $\psi(G;x,\yy,\frac{y-x}{x})$ from $G$. We have
  \[ \psi\Big(G;x,\yy,\frac{y-x}{x}\Big) = \xi(G;x,\yy, (y-x)\yy) \]
  by Lemma~\ref{lem:psi_xi}. An easy calculation converts the
  conditions on $x,\yy, y$ into conditions on $x, y, z$ and yields the
  statement of the theorem.

  Now assume that we are able to evaluate $\psi$ at some fixed
  $(x,y,z)\in \Q^3$, i.e.\ given $G$ we can compute $\psi(G;x,y,z)$.
  Then Theorem~\ref{thm:p2p} allows us to evaluate $\psi$ at
  $(x,y',z)$ for infinitely many different $y'=(1+y)^k-1$ provided
  that $|1+y|\neq 0$ and $|1+y|\neq 1$. As $\psi$ is a polynomial,
  this enables interpolation in $y$ and eventually gives us the
  ability to evaluate $\psi$ at $(x,y',z)$ for \emph{any}\ $y'\in \Q$.
  In particular, being able to evaluate $\psi$ at
  $(x,\yy,\frac{y-x}{x})$, $\yy\in \Q\setminus \{-2, -1, 0\}$, implies
  the ability to evaluate it at $(x, -1, \frac{y-x}{x})$.
\end{proof}

\bibliographystyle{alpha}
\bibliography{literatur}

\end{document}